\documentclass[a4paper,12pt]{report}
\usepackage{amsfonts,amssymb}
\usepackage{theorem}
\newtheorem{proposition}{Proposition}

\newtheorem{theorem}{Theorem}
{\theorembodyfont{\rmfamily}
\newtheorem{definition}{Definition}}
{\theorembodyfont{\rmfamily}
\newtheorem{example}{Example}}
{\theorembodyfont{\rmfamily}
\newtheorem{remark}{Remark}}
{\theorembodyfont{\rmfamily}
\newtheorem{conjecture}{Conjecture}}

\newsavebox{\proofbox}
\savebox{\proofbox}{%
  \begin{picture}(7,7)\put(0,0){\framebox(7,7){}}\end{picture}%
}
\newenvironment{proof}{%
  \list{}{\leftmargin0pt
    \rightmargin\leftmargin}%
  \item[]{\hspace*{1em}\it Proof.\ }%
}
{\hspace*{\fill}{\usebox{\proofbox}}\endlist}
\begin{document}
\frenchspacing
\vskip 3cm
\begin{center}
\LARGE Research talk given for the CALF-seminar, Oxford 23.6.2005
\end{center}
\vskip 2cm
\begin{center}
\begin{tabular}{l}
\Large On the Classification of Generalized Cartan Matrices of Rank $3$
\end{tabular}
\end{center}
\vskip6cm
\begin{flushright}
\begin{tabular}{l}
Kyriakos Papadopoulos \\
kxp878@bham.ac.uk
\end{tabular}
\end{flushright}
\vskip2cm
\begin{flushright}
\begin{tabular}{l}
\end{tabular}
\end{flushright}
\chapter*{Foreword}\pagenumbering{arabic}
In this talk we describe very briefly how Gritsenko and Nikulin classified in \cite{3}
the generalised Cartan matrices of rank 3, of elliptic type, which
are twisted to symmetric generalised Cartan matrices. We will then introduce a conjecture on this
classification and propose an algorithm for attacking it.

The author would like to thank prof. V.V. Nikulin for giving him
the opportunity to study this subject.
\tableofcontents

\section[Reflection Groups of Integral Hyperbolic Lattices and Generalised Cartan Matrices]{Reflection Groups of Integral Hyperbolic Lattices and Generalised Cartan Matrices\footnotemark}
\footnotetext{The presentation of this section is based on~\cite{1},~\cite{2}and~\cite{3}.}
In this section we will give all the basic definitions on generalised Cartan matrices, and show how they
are related to reflection groups of integral hyperbolic lattices.

\begin{definition}
For a countable set of indices $I,$ a finite-rank matrix $A=(a_{ij})$ is called {\em generalised Cartan matrix}, if and
only if:
\begin{enumerate}

\item $a_{ii}=2\,;$

\item $a_{ij} \in \mathbb{Z}^-,\, i \neq j$ and

\item $a_{ij}=0 \Rightarrow a_{ji}=0\,.$

\end{enumerate}

We consider such a matrix to be {\em indecomposable}, i.e. there
does not exist a decomposition $I=I_1\, \cup \, I_2,$ such that
$I_1 \neq \emptyset,\,I_2 \neq \emptyset$ and $a_{ij}=0,$ for $i
\in I_1,\,j \in I_2.$
\end{definition}

\begin{definition}
A generalised Cartan matrix $A$ is {\em symmetrizable}, if there exists invertible diagonal matrix
$D=\textrm{diag}(\ldots\epsilon_i \ldots)$ and symmetric matrix $B=(b_{ij}),$ such that:
\[A=DB \textrm{ or } (a_{ij})=(\epsilon_i\,b_{ij}),\] where $\epsilon_i \in \mathbb{Q},\,\epsilon_i >0,\,b_{ij} \in \mathbb{Z}^-,\,
b_{ii} \in 2 \mathbb{Z}$ and $b_{ii}>0.$

The matrices $D$ and $B$ are defined uniquely, up to a multiplicative constant, and $B$ is called {\em symmetrised
generalised Cartan matrix}.
\end{definition}

\begin{remark}
A symmetrizable generalised Cartan matrix $A=(a_{ij})$ and its symmetrised generalised Cartan matrix $B=(b_{ij})$
are related as follows:
\[(a_{ij})= \Bigl({{2\,b_{ij}} \over {b_{ii}}}\Bigr),\] where $b_{ii} | 2\,b_{ij}.$
\end{remark}

\begin{definition}
A symmetrizable generalised Cartan matrix $A$ is called {\em hyperbolic}, if its symmetrised generalised Cartan matrix
$B$ has exactly one negative square (or $A$ has exactly one negative eigenvalue).
\end{definition}

Let us now recall the definition for {\em hyperbolic integral
quadratic form} $S:$
\begin{definition}
A {\em hyperbolic integral symmetric bilinear form} (or
hyperbolic integral quadratic form) on a finite rank free $\mathbb{Z}-$
module $M,$ of dimension $n,$ over the ring of integers,
is a map:
\[S:M \times M \to \mathbb{Z},\]
satisfying the following conditions:

\begin{enumerate}

\item $S(\alpha\,m_1+\beta\,m_2,m_3)=\alpha\,S(m_1,m_3)+\beta\,S(m_2,m_3);$

\item $S(m_3,\alpha\,m_1+\beta\,m_2)=\alpha\,S(m_3,m_1)+\beta\,S(m_3,m_2);$

\item $S(m_1,m_2)=S(m_2,m_1)$ (symmetry) and

\item signature=$(n,1);$ i.e. in a suitable basis, the corresponding matrix of $S$ is a diagonal matrix
with $n$ positive squares and one negative square in the diagonal,
\end{enumerate}
where $\alpha,\,\beta \in \mathbb{Z}$ and $m_1,m_2,m_3 \in M.$
\end{definition}

We now consider an {\em integral hyperbolic lattice} $(M,S),$ i.e. a pair of a free $\mathbb{Z}-$module $M$ and
a hyperbolic integral symmetric bilinear form $S.$

By considering the corresponding cone:
\[V(M)=\{x \in M \otimes \mathbb{R}:(x,x)<0\},\]
and by choosing its half-cone $V^+(M),$ we can define the corresponding {\em hyperbolic (or Lobachevskii) space}:
\[\Lambda^+(M)=V^+(M)/\mathbb{R}^{++},\] as a section (slice) of the cone (a cone is by definition the set of rays with
origin at zero), by a hyperplane.

After defining the hyperbolic space, we can work in hyperbolic geometry, by defining the {\em distance} $\rho$ between
two points $\mathbb{R}^{++}x$ and $\mathbb{R}^{++}y$ in $\Lambda^+(M),$ as follows:
\[cosh\,\rho\,(\mathbb{R}^{++}x,\mathbb{R}^{++}y)= \frac{-S(x,y)}{\sqrt{S(x,x)S(y,y)}}\]
Obviously these two points in hyperbolic space are rays in the half-cone $V^+(M).$

\begin{remark}
By definition, when we use signature $(n,1),$ the square of a vector which
is outside the cone $V(M)$ is strictly greater than zero and the square of a vector which is
inside the cone is strictly negative. Furthermore, if a vector lies
on the surface of the cone, then its square is zero. Obviously, we put a minus sign in the numerator
of the definition of distance in hyperbolic space, because the hyperbolic cosine should be always positive.
\end{remark}

Now, each element $\alpha_i \in M \otimes \mathbb{R},$ with $(\alpha_i,\alpha_i)>0,$
defines the half spaces:
\[H_{\alpha_i}^+ = \{\mathbb{R}^{++}x \in \Lambda^{+}(M):(x,\alpha_i) \leq 0\},\]
\[H_{\alpha_i^-}^+ = \{\mathbb{R}^{++}x \in \Lambda^{+}(M):(x,\alpha_i) > 0\},\]
which are bounded by the hyperplane:
\[H_{\alpha_i}=\{\mathbb{R}^{++}x \in \Lambda^+(M):(x,\alpha_i)=0\},\]
where $\alpha_i \in M \otimes \mathbb{R}$ is defined up to multiplication
on elements of $\mathbb{R}^{++}.$ The hyperplane $H_{\alpha_i}$ is also
called {\em mirror of symmetry}.

Let us denote by $O(M)$ the {\em group of automorphisms}, which preserves
the cone $V(M).$ Its subgroup $O^+(M) \subseteq O(M)$ is of index $2,$ and
fixes the half-cone $V^+(M).$ Furthermore, $O^+(M)$ is discrete in $\Lambda^+(M),$
and has fundamental domain of finite volume.

\begin{proposition}
The index $[O(M):O^+(M)]$ is always equal to $2.$
\end{proposition}

\begin{proof}
The group $O(M)$ acts on two elements which are the half cones $V^+$ and its
opposite $V^-$. So, the kernel of this action, $O^+(M),$ cannot have
index greater than $2.$ Also, the element $-1,$ from $O(M),$ changes $V^+$ with $V^-$.
Thus the index is exactly equal to $2$.
\end{proof}

\begin{definition}
For $(\alpha_i,\alpha_i)>0,$ by $s_{a_i} \in O^+(M)$ we define {\em reflection}
in a hyperplane $H_{\alpha_i},$ of $\Lambda^+(M),$ as follows:
\[s_{\alpha_i}(x)=x- \frac{2(x,\alpha_i)}{(\alpha_i,\alpha_i)}\alpha_i,\]
where $x \in M$ and $\alpha_i \in M.$
\end{definition}

\begin{remark}
Why this equation, that we gave for reflection in our hyperbolic lattice, works? Answer: because of
the following two facts:

\begin{enumerate}

\item for $x=\alpha_i,$ we get that $s_{\alpha_i}(\alpha_i)=-\alpha_i$ and

\item for $x$ perpendicular to $\alpha_i,$ we have that $s_{\alpha_i}(x)=x,$

both of which show that our formula cannot work if we omit number $2$ from the numerator.
\end{enumerate}

An obvious remark is that a reflection $s_{a_i}$ changes place between the half-spaces $H_{a_i}^+$ and $H_{a_i^-}^+.$
\end{remark}

Going a bit further, if an orthogonal vector $\alpha_i \in M,\,(\alpha_i,\alpha_i)>0,$ in a hyperplane $H_{a_i}$ of $\Lambda^+(M),$ is a
{\em primitive root}, i.e. its coordinates are coprime numbers, then:
\[\frac{2(M,\alpha_i)}{(\alpha_i,\alpha_i)}\alpha_i \subseteq M \Leftrightarrow (\alpha_i,\alpha_i)|\,2(M,\alpha_i)\]

\begin{definition}
Any subgroup of $O(M)$ (the corresponding {\em discrete} group of motions of $\Lambda(M)$), generated
by reflections, is called {\em reflection group}.

We denote by $W(M)$ the subgroup of $O^+(M)$ generated by {\em all} reflections of $M,$  of elements with positive
squares (always for signature $(n,1)$).

We will also denote by $W$ the subgroup of $W(M)$ generated by reflections in a set of elements of $M.$
Obviously, $W \subseteq W(M) \subseteq O^+(M)$ is a subgroup of finite index.
\end{definition}

\begin{definition}
A lattice $M$ is called {\em reflective}, if index $[O(M):W(M)]$ is finite. In other words, $W(M)$ has fundamental
polyhedron of finite volume, in $\Lambda(M).$
\end{definition}

Talking a bit more about lattices, we consider again our integral
hyperbolic lattice $S: M \times M \to \mathbb{Z}.$ For $m \in
\mathbb{Q}$ we denote by $S(m)$ the lattice which one gets, if
multiplying $S$ by $m.$ If $S$ is reflective, then $S(m)$ is
reflective, too.

Furthermore, $S$ is called {\em even lattice}, if $S(x,x)$ is even, $x \in M. $ Otherwise, $S$ is
called $odd.$

Last, but not least, $S$ is called {\em primitive lattice} (or {\em even primitive}), if
$S(\frac{1}{m})$ is not lattice (or even lattice), $m \in \mathbb{N},\,m \ge 2.$

\begin{definition}
A convex polyhedron $\cal M,$ in $\Lambda(M),$ is an intersection:
\[{\cal{M}}=\bigcap_{\alpha_i} H_{\alpha_i}^+\]
of several half-spaces orthogonal to elements $\alpha_i \in M,\,(\alpha_i,\alpha_i)>0.$

This convex polyhedron is the fundamental chamber for a reflection group $W(M),$ with
reflections generated by (primitive) roots, in $M,$ each of them orthogonal to exactly
one side of this polyhedron. Moreover, we get this chamber if we remove all
mirrors of reflection, and take the connencted components of the complement (with
the boundary). The fundamental chamber acts {\em simply transitively}, because
if we consider an element $w \in W(M),$ then $w({\cal{M}}_1)={\cal{M}}_2,$ in other
words it fills in our (hyperbolic) space with congruent polyhedra.

The polyhedron $\cal{M}$ belongs to the cone:
 \[\mathbb{R}^+{\cal{M}}=\{ x \in V^+(M):(x,\alpha_i) \leq 0\},\]
where $\alpha_i \in P({\cal{M}})=\{\alpha_i: i \in I\}$; set of orthogonal vectors
to ${\cal{M}}$, where exactly one element $\alpha_i$ is orthogonal to each face of ${\cal{M}}.$

We say that $P({\cal{M}})$ is {\em acceptable}, if each of its elements is a (primitive) root,
which is perpendicular to exactly one side of a convex polyhedron, in $\Lambda^+(M).$

Also, ${\cal{M}}$ is {\em non-degenerate}, if it contains a non-empty open subset of $\Lambda^+(M)$
and {\em elliptic}, if it is a convex envelope of a finite set of points in $\Lambda^+(M)$ or at infinity
of $\Lambda^+(M).$
\end{definition}

Let us now return back to the theory for generalised Cartan
matrices, and relate it with the material that we introduced for
reflection groups of integral hyperbolic lattices.

\begin{definition}
An indecomposable hyperbolic generalised Cartan matrix A is equivalent to a triplet:
\[A \sim (M,W,P(\mathcal{M})),\] where $S : M \times M \to \mathbb{Z}$ is a hyperbolic integral symmetric bilinear
form, $W \subseteq W(M) \subseteq O^+(M),$ $W$ is a subgroup of reflections in a set of elements
of $M,$ with positive squares, $W(M)$ is the subgroup of reflections in all elements of $M$ (with positive
squares), $O^+(M)$ is the group of automorphisms, which fixes the half-cone $V^+(M),$ and $P(\mathcal{M})=
\{\alpha_i: i \in I\},\,(\alpha_i,\alpha_i)>0,$ $A=\Bigl( 2{{(\alpha',\alpha)} \over {(\alpha ,\alpha)}} \Bigr),$ where
$\alpha,\,\alpha' \in P(\mathcal{M}),\,\mathcal{M}$ is a locally finite polyhedron in $\Lambda^+(M),\,
{\cal{M}}=\bigcap_{\alpha_i} H_{\alpha_i}^+$ and $H_{\alpha_i}^+ = \{\mathbb{R}^{++}x \in \Lambda^{+}(M):(x,\alpha_i) \le 0\}.$

The triplet $(M,W,P(\mathcal{M}))$ is called {\em geometric realisation of A}.
\end{definition}

Let us now consider $\lambda(\alpha) \in \mathbb{N},\,\alpha \in P(\mathcal{M})$ and gcd$(\{\lambda(\alpha):\alpha \in P(\mathcal{M})\})=1.$

$\tilde{{\alpha}}=\lambda(\alpha)\,\alpha,\,(\tilde{{\alpha}},\tilde{{\alpha}})|\,2(\tilde{{\alpha}}',\tilde{{\alpha}}) \Leftrightarrow
\lambda(\alpha)(\alpha,\alpha)|\,2 \lambda(\alpha')(\alpha',\alpha)$

Then, $\tilde{A}=\Bigl(2(\tilde{{\alpha}}',\tilde{{\alpha}}) /
(\tilde{{\alpha}},\tilde{{\alpha}})\Bigl)=\Bigl(2
\lambda(\alpha')({\alpha}',{\alpha}) /
\lambda(\alpha)({\alpha},{\alpha})\Bigr),\,\alpha,\alpha' \in
P(\mathcal{M}),$ where $\tilde{A}$ is a {\em twisted, to} A,
hyperbolic generalised Cartan matrix and $\lambda(\alpha)$ are
called {\em twisted coefficients of} $\alpha.$

Also, $\tilde{A}=(\tilde{M},\tilde{W},\tilde{P}(\tilde{\mathcal{M}}))=(M \supseteq[\{\lambda(\alpha)\,\alpha: \alpha \in P(\mathcal{M})\}\,],W,
\{\lambda(\alpha)\,\alpha:\alpha \in P(\mathcal{M})\}).$ In other words, $W$ and $\mathcal{M}$ are the same for $\tilde{A}$ and $A.$

Obviously, $A$ is {\em untwisted}, if it cannot be twisted to any generalised Cartan matrix, different from itself.

Our purpose is to work on hyperbolic generalised Cartan matrices of {\em elliptic type}, so we need to introduce
some more material:

\begin{definition}
Let $A$ be a hyperbolic generalised Cartan matrix, $A \sim (M,W,P(\mathcal{M})).$ We define the {\em group of symmetries of} $A$ (or
$P(\mathcal{M}))$ as follows:
\[\textrm{Sym }(A)=\textrm{Sym }(P(\mathcal{M}))=\{g \in O^+(M):g(P(\mathcal{M}))=P(\mathcal{M})\}\]
\end{definition}

\begin{definition}
A hyperbolic generalised Cartan matrix $A$ has {\em restricted
arithmetic type}, if it is not empty and the semi-direct product
of $W$ with Sym$(A),$ which is equivalent to the semi-direct
product of $W$ with Sym$(P(\mathcal{M})),$ has finite index in
$O^+(M).$
\end{definition}

\begin{remark}
For $W=(w_1,s_1),\,S=(w_2,s_2)\,\mathbb{Z}-$modules, the
semi-direct product of $W$ with $S$ is equal to $((s_2 w_1)w_2,s_1
s_2).$
\end{remark}

\begin{definition}
A hyperbolic generalised Cartan matrix $A$ has {\em lattice Weyl vector}, if there exists $\rho \in M \otimes \mathbb{Q},$
such that:
\[(\rho,\alpha)=-(\alpha,\alpha)/2,~\alpha \in P(\mathcal{M})\]

Additionally, $A$ has {\em generalised lattice Weyl vector}, if there exists $\bold{0} \neq \rho \in M \otimes \mathbb{Q},$ such that
for constant $N >0:$
\[0 \le -(\rho,\alpha) \le N\]

We can think of $\rho$ in $\Lambda^+(M)$ as being the centre of the inscribed circle to $\mathcal{M},$ where $\mathcal{M}$ is
the fundamental chamber of a reflection group $W.$
\end{definition}

And we can now give the last, and very important, definition of this introductory section:

\begin{definition}
A hyperbolic generalised Cartan matrix $A$ has {\em elliptic type}, if it has restricted arithmetic type and generalised
lattice Weyl vector $\rho,$ such that $(\rho,\rho)<0.$ In other words, $[O(M):W]<\infty$ or vol$(\mathcal{M})<\infty$ or
$P(\mathcal{M})< \infty.$
\end{definition}

\section[The Classification of Generalised Cartan Matrices of Rank 3, of Elliptic Type, with the Lattice Weyl
Vector, which are twisted to Symmetric Generalised Cartan Matrices]{The Classification of Generalised Cartan Matrices
of Rank 3, of Elliptic Type, with the Lattice Weyl Vector, which are twisted to Symmetric Generalised Cartan Matrices\footnotemark}
\footnotetext{The presentation of this section is based on~\cite{3}.}
In this section we are going to describe very briefly how Gritsenko and Nikulin classified in \cite{3}
the generalised Cartan matrices of rank 3, of elliptic type (so they have the generalised lattice Weyl vector), which
are twisted to symmetric generalised Cartan matrices. We will then introduce a conjecture on this
classification and propose an algorithm for attacking it.

Let $A$ be a generalised Cartan matix of elliptic type, twisted to a symmetric generalised Cartan matrix $\tilde{A}.$This
automatically implies that $\tilde{A}$ is of elliptic type, too. Let also $G(A)=(M,W,P(\mathcal{M}))$ be the geometric realisation of $A,$ where rank of $A$ is equal to $3.$

Furthermore, for $\alpha \in P(\mathcal{M}),$ we set:
\[\alpha=\lambda(\alpha)\,\delta(\alpha),\] where $\lambda(\alpha) \in \mathbb{N}$ are the twisted coefficients of $\alpha,$ and
$(\delta(\alpha),\delta(\alpha))=2.$ So, $\tilde{P}(\mathcal{M})=\{\delta(a)=\alpha/\lambda(\alpha): \alpha \in P(\mathcal{M})\}.$

Notation: from now on, $\delta(\alpha_i)=\delta_i$ and $\lambda(\alpha_i)=\lambda_i.$

Now, $A$ and its geometric realisation are equivalent to a $(1+[n/2])\times n$ matrix $G(A):$
\[\textrm{1st raw: }\lambda_1,\ldots,\lambda_n\]
\[\textrm{(i+1)th raw: }-(\delta_1,\delta_{1+i}),\ldots,-(\delta_n,\delta_{n+1});~1 \le i \le [n/2]\]
\[\textrm{j\,th column: }(\lambda_j,(\delta_j,\delta_{j+1}),\ldots,(\delta_j,\delta_{j+[n/2]}))^t;~1 \le j \le n(mod\,n)\]

Let us illustrate this by giving a specific example:

\begin{example}
Let $\delta_1,\,\delta_2,\,\delta_3,\,\delta_4,\,\delta_5$ be elements with positive squares, each of them orthogonal
to exactly one side of a convex polytope, of five sides, in hyperbolic space. So, $1 \le i \le [5/2]$ and:
\[G(A)=  \left(\begin{array}{rrrrr} \lambda_1 & \lambda_2 & \lambda_3 & \lambda_4 & \lambda_5 \\ \delta_1\,\delta_2 & \delta_2\,\delta_3 &
\delta_3\,\delta_4 & \delta_4\,\delta_5 & \delta_5\,\delta_1 \\ \delta_1\,\delta_3 & \delta_2\,\delta_4 & \delta_3\,\delta_5 & \delta_4\,\delta_1 &\delta_5\,\delta_2\end{array}\right)\]
\end{example}

Our problem is to find all matrices $G(A),$ having the lattice Weyl vector $\rho,$ $\rho \in M \otimes \mathbb{Q},$ such that
$(\rho,\alpha)=-(\alpha,\alpha)/2 \Leftrightarrow (\rho,\delta_i)=-\lambda_i,\,i=1,\ldots,n.$ The answer has been presented
by Gritsenko and Nikulin in the following theorem (1.2.1. from \cite{3}):

\begin{theorem}
All geometric realisations $G(A),$ of hyperbolic generalised Cartan matrices $A$ of rank $3,$ of elliptic type, with
the lattice Weyl vector, which are twisted to symmetric generalised Cartan matrices, all twisting coefficients $\lambda_i$
satisfy:
\[\lambda_i \le 12\]
are given in Table 1 (from \cite{3}).
\end{theorem}

\begin{remark}
Table 1 (from \cite{3}) gives $60$ matrices. $7$ of them are of the compact case (they represent a convex polytope
of finite volume in hyperbolic space), the $4$ of which are untwisted. The rest $53$ matrices are of the non-compact
case.

All these matrices can be found in pages 166-168 , from \cite{3}.
\end{remark}

The main aim of this section is to discuss about the conjecture that follows the above theorem, which we give
below:

\begin{conjecture}
Table 1 (from \cite{3}) gives the complete list of hyperbolic generalised Cartan matrices $A$ of rank $3,$ of
elliptic type, with the lattice Weyl vector, which are twisted to symmetric generalised Cartan matrices.

In other words, one can drop the inequality $\lambda_i \le 12$ from theorem 1.2.1., from $\cite{3}.$
\end{conjecture}

Gritsenko and Nikulin give, in the same paper, the following arguments for supporting the conjecture:
\begin{enumerate}

\item The number of all hyperbolic generalised Cartan matrices of elliptic type, with the lattice Weyl vector,
is finite, for rank greater than or equal to $3.$

So, there exists an absolute constant $m,$ such that $\lambda_i \le m.$

\item Calculations were done for all $\lambda_i \le 12,$ but result has only matrices with all $\lambda_i \le 6.$

So, there do not exist new solutions between $6$ and $12.$

\end{enumerate}

We will now propose a way of attacking the mentioned conjecture, in four
steps, giving the complete solution for the first proposed step.

FIRST STEP: We consider a triangle in hyperbolic space, with sides $a,\,b$ and $c.$ The angle between $a$ and $b$ is $\pi/2,$
between $a$ and $c$ is $\pi/3$ and between $c$ and $b$ is $0$ radians, that is, the vertex which is created from
the intersection of $c$ and $b$ is at infinity of our space. This triangle will be the fundamental chamber for
reflection in $\Lambda^+(M),$ and the reflections will cover all space, tending to infinity.

In the proof of theorem 1.2.1, the authors gave the following relations, which we will use:
\[0\le(\delta_1,\delta_2)\le 2,~0\le(\delta_1,\delta_3)< 14,~0\le(\delta_2,\delta_3)\le 2~~(*),\] for
$\delta_1,\,\delta_2,\,\delta_3 \in \tilde{P}(\mathcal{M})$ being
orthogonal vectors to three consecutive sides, of a polygon $A_1
\ldots A_n$ in $\Lambda^+(M),$ namely $A_1A_2,\,A_2A_3,\,A_3A_4.$

Our suggestion is to find all these $\delta_1,\,\delta_2,\,\delta_3$ satisying $(*),$ for the group generated
by reflections in $\Lambda^+(M),$ with fundamental chamber the triangle with sides $a,\,b,\,c.$

We fix $\delta_2=a.$ Then, we have the following possibilities for
$\delta_1:$

\[\delta_1=c\]
\[s_c(b)=b-2((b,c)/c^2)c=b+2c\]
\[s_{b+2c}(a)=a-2((a,b+2c)/(b+2c)^2)(b+2c)=a+2b+4c\]
etc.

In other words:
\[\delta_1=na+(n+1)b+2(n+1)c\]

Now, the posibilities for $\delta_3$ are:
\[s_b(c)=c-2((c,b)/b^2)b \Rightarrow \delta_3=2b+c\]
\[s_{2b+c}(-b)=-b-2((b,2b+c)/(2b+c)^2)(2b+c) \Rightarrow \delta_3=3b+2c\]
Also:
\[s_{3b+2c}(a)=a- 2((a,3b+2c)/(3b+2c)^2)(3b+2c)=a+6b+4c\]
\[s_{a+6b+4c}(-3b-2c)=(-3b-2c)-2((-3b-2c,a+6b+4c)/(a+6b+4c)^2)(a+6b+4c)=\]
\[2a+9b+6c\]
etc.

So,
\[\delta_3=(n+1)a+(3n+6)b+(2n+4)c\]

\begin{remark}
We fixed $\delta_2=a,$ and we looked for possible $\delta_1$ and $\delta_3,$ such that the angle
between $\delta_2$ and $\delta_1$ is accute; the same for the angle between $\delta_2$ and $\delta_3.$
\end{remark}

Below, we give the list of all possible $\delta_1$ and $\delta_3,$ for $\delta_2=a,$ such that
$(\delta_1,\delta_3)<14$ (in the right hand side we give $\delta_3$ and in the left $\delta_1):$

\[c,~b\]

\[b+2c,~b\]
\[a+2b+4c,~b\]
\[2a+3b+6c,~b\]
\[3a+4b+8c,~b\]
\[4a+5b+10c,~b\]
\[5a+6b+12c,~b\]

\[c,~2b+c\]
\[c,~3b+2c\]

\[c,~a+6b+4c\]
\[c,~2a+9b+6c\]
\[c,~3a+12b+8c\]

\[b+2c,~2b+c\]
\[a+2b+4c,~2b+c\]
\[2a+3b+6c,~2b+c\]
\[3a+4b+8c,~2b+c\]

\[b+2c,~3b+2c\]
\[a+2b+4c,~3b+2c\]
\[2a+3b+6c,~3b+2c\]

\[b+2c,~a+6b+4c\]

\[b+2c,~2a+9b+6c\]

We now fix $\delta_2=b,$ and give the list of all possible $\delta_1$ and $\delta_3,$ for $\delta_2=b,$ such that
$(\delta_1,\delta_3)<14$ (in the right hand side we give $\delta_3$ and in the left $\delta_1):$

\[c,~a\]

\[b+2c,~a\]
\[2b+3c,~a\]
\[3b+4c,~a\]
\[4b+5c,~a\]
\[5b+6c,~a\]
\[6b+7c,~a\]
\[7b+8c,~a\]
\[8b+9c,~a\]
\[9b+10c,~a\]
\[10b+11c,~a\]
\[11b+12c,~a\]
\[12b+13c,~a\]

\[c,~a+c\]
\[c,~2a+b+2c\]
\[c,~3a+2b+3c\]
\[c,~4a+3b+4c\]
\[c,~5a+4b+5c\]
\[c,~6a+5b+6c\]
\[c,~7a+6b+7c\]
\[c,~8a+7b+8c\]
\[c,~9a+8b+9c\]
\[c,~10a+9b+10c\]
\[c,~11a+10b+11c\]
\[c,~12a+11b+12c\]
\[c,~13a+12b+13c\]
\[c,~14a+13b+14c\]
\[c,~15a+14b+15c\]

\[b+2c,~a+c\]
\[b+2c,~2a+b+2c\]
\[b+2c,~3a+2b+3c\]
\[b+2c,~4a+3b+4c\]
\[b+2c,~5a+4b+5c\]
\[b+2c,~6a+5b+6c\]
\[b+2c,~7a+6b+7c\]

\[2b+3c,~a+c\]
\[2b+3c,~2a+b+2c\]
\[2b+3c,~3a+2b+3c\]
\[2b+3c,~4a+3b+4c\]
\[2b+3c,~5a+4b+5c\]

\[3b+4c,~a+c\]
\[3b+4c,~2a+b+2c\]
\[3b+4c,~3a+2b+3c\]

\[4b+5c,~a+c\]
\[4b+5c,~2a+b+2c\]
\[4b+5c,~3a+2b+3c\]

\[5b+6c,~a+c\]
\[5b+6c,~2a+b+2c\]

\[6b+7c,~a+c\]
\[6b+7c,~2a+b+2c\]

\[7b+8c,~a+c\]
\[7b+8c,~2a+b+2c\]

\[8b+9c,~a+c\]

\[9b+10c,~a+c\]

\[10b+11c,~a+c\]

\[11b+12c,~a+c\]

\[12b+13c,~a+c\]

\[13b+14c,~a+c\]

\[14b+15c,~a+c\]

Last, we fix $\delta_2=c,$ and give the list of all possible $\delta_1$ and $\delta_3,$ for $\delta_2=c,$ such that
$(\delta_1,\delta_3)<14$ (in the right hand side we give $\delta_3$ and in the left $\delta_1):$

\[b,~a\]

\[2b+c,~a\]
\[3b+2c,~a\]
\[4b+3c,~a\]
\[5b+4c,~a\]
\[6b+5c,~a\]
\[7b+6c,~a\]
\[8b+7c,~a\]
\[9b+8c,~a\]
\[10b+9c,~a\]
\[11b+10c,~a\]
\[12b+11c,~a\]
\[13b+12c,~a\]
\[14b+13c,~a\]

\[b,~2a+b+2c\]

\[b,~8a+6b+9c\]
\[b,~12a+9b+14c\]
\[b,~16a+12b+19c\]

\[2b+c,~2a+b+2c\]
\[2b+c,~3a+2b+3c\]
\[2b+c,~4a+3b+4c\]
\[2b+c,~5a+4b+5c\]
\[2b+c,~6a+5b+6c\]
\[2b+c,~7a+6b+7c\]
\[2b+c,~8a+7b+8c\]
\[2b+c,~9a+8b+9c\]
\[2b+c,~10a+9b+10c\]
\[2b+c,~11a+10b+11c\]

\[3b+2c,~2a+b+2c\]
\[3b+2c,~3a+2b+3c\]
\[3b+2c,~4a+3b+4c\]
\[3b+2c,~5a+4b+5c\]

\[4b+3c,~2a+b+2c\]
\[4b+3c,~3a+2b+3c\]

\[5b+4c,~2a+b+2c\]
\[5b+4c,~3a+2b+3c\]

\[6b+5c,~2a+b+2c\]

SECOND STEP: So, we have found $115$ triples of elements
$\delta_1,\,\delta_2,\,\delta_3,$ as we described above, and we
now need to find, for each case separately, the correspoding
twisting coefficients, $\lambda_1,\,\lambda_2,\,\lambda_3.$ So, we
introduce three new elements,
$\tilde{\delta_1}=\lambda_1\,\delta_1,\,\tilde{\delta_2}=\lambda_2\,\delta_2,\,\tilde{\delta_3}=
\lambda_3,\delta_3$

As we mentioned in the first section, any $\tilde{\delta_i}$ and $\tilde{\delta_j}$ should satisfy the relations:
\[\tilde{\delta_i}^2|\,2(\tilde{\delta_i},\tilde{\delta_j}) \Rightarrow 2\lambda_i|\,2\lambda_j(\delta_i,\delta_j)~(**)\]
\[\tilde{\delta_j}^2|\,2(\tilde{\delta_i},\tilde{\delta_j})  \Rightarrow 2\lambda_j|\,2\lambda_i(\delta_i,\delta_j)\]

Now, let us introduce some more notation; so, by $\nu_p(\lambda),$ we denote the power of the
prime number $p,$ in the prime factorisation of the natural number $\lambda.$ For example, $\nu_3(21)=1$
and $\nu_2(21)=0.$

Applying this notation to our case, we get that for our
$\tilde{\delta_1},\,\tilde{\delta_2},\,\tilde{\delta_3},$ the
following conditions should be satisfied:

\[|\nu_p(\lambda_1)-\nu_p(\lambda_2)| \le \nu_p(g_{12})\]
\[|\nu_p(\lambda_1)-\nu_p(\lambda_3)| \le \nu_p(g_{13})\]
\[|\nu_p(\lambda_2)-\nu_p(\lambda_3)| \le \nu_p(g_{23})\]

where $g_{12},\,g_{13},\,g_{23}$ are elements of the Gram matrix of $\delta_1,\,\delta_2,\,\delta_3,$ with
diagonal equal to $2$ and (by definition) $g_{ij}=g_{ji}.$

So, the information that we have is quite enough, for calculating $\lambda_1,\,\lambda_2,\,\lambda_3,$
for each of the $115$ triples that we calculated in (i), simply by working on the Gram matrix for each case.

Another way to do this calculations is via programming. Here we introduce a programme in MAPLE (with
some explanation following the programme):

with(numtheory);

[GIgcd, bigomega, cfrac, cfracpol, cyclotomic, divisors, factorEQ, factorset,

fermat, imagunit, index, integral\_basis, invcfrac, invphi, issqrfree,

jacobi, kronecker, lambda, legendre, mcombine, mersenne, migcdex, minkowski,

mipolys, mlog, mobius, mroot, msqrt, nearestp, nthconver, nthdenom,

nthnumer, nthpow, order, pdexpand, phi, pi, pprimroot, primroot, quadres,

rootsunity, safeprime, sigma, sq2factor, sum2sqr, tau, thue]

 x\_y:=2: x\_z:=4: y\_x:=2: y\_z:=2: z\_x:=4: z\_y:=2:

count:=0:

for x in divisors(lcm(x\_y*x\_z))do

 for y in divisors(lcm(y\_x*y\_z)) do

for z in divisors(lcm(z\_x,z\_y)) do

if modp(x\_y*y,x)=0 and

modp(x\_z*z,x)=0 and

 modp(y\_x*x,y)=0 and

modp(y\_z*z,y)=0 and

 modp(z\_x*x,z)=0 and

 modp(z\_y*y,z)=0 and igcd(x,y,z)=1

 then print(x,y,z);

count:=count+1;

 fi;

od;

 od;

od;

 printf ("Number of solutions 

1, 1, 1

1, 1, 2

1, 2, 1

1, 2, 2

1, 2, 4

2, 1, 1

2, 1, 2

2, 2, 1

4, 2, 1

Number of solutions 9

This programme works as follows; for each $\delta_1,\,\delta_2,\,\delta_3$ we apply $(**)$ and get six relations of the type:
\[x | x\_y * y\]
\[x | x\_z * z\]
\[y | y\_x * x\]
\[y | y\_z * z\]
\[z | z\_x * x\]
\[z | z\_y * y\]
where gcd$(x,y,z)=1.$ It is enough to examine the values of $x,$ which belong to the set $S_x,$ of the divisors of
the lcm$(x\_y*x\_z).$ In other words, lcm$(x,y,z)=x.$ The same for $y$ and $z,$ we have the sets $S_y$ and $S_z$
respectively, and check for which triples $(x,y,z)$ (with $x,\,y,\,z$ be longing to $S_x,\,S_y,\,S_z$ respactively) the conditions hold.

THIRD STEP: The third step is to find for each triple $\delta_1,\,\delta_2,\,\delta_3$ the lattice
Weyl vector. And we already mentioned, in the first section, that the following relation should be satisfied:
\[(\rho,\delta_i)=-\lambda_i,~1 \le i \le 3\]

Furthermore, it should be always true that $(\rho,\rho)<0.$

We are optimistic that we will have a computer
programme which can do this calculation, soon.

FOURTH STEP: In this step we will try to find more elements $\delta_i,$ which will be orthogonal to sides of polytopes in $\Lambda^+(M);$ polytopes to which
$\delta_1,\,\delta_2,\,\delta_3$ were orthogonal (each one perpendicular to a side of the polytope, respectively).

So, for each case (from steps (i)-(iii)) we should examine if one more side exists, where a new element $\delta_4$ is orthogonal to
this side. It will be vital to prove the existence of the lattice Weyl vector.

What we already know is that the determinant of the Gram matrix of $\delta_1,\,\delta_2,\,\delta_3,\,\delta_4$ should be equal to zero,
because $\delta_1,\,\delta_2,\,\delta_3$ should come from a $3-$D lattice, and so the four elements $\delta_1,\,\delta_2,\,\delta_3,\,\delta_4$ should
be linearly independent.

Also, $(**)$ should take the form:
\[(\lambda_4\,\delta_4)^2|\,2(\lambda_4\,\delta_4,\lambda_i\,\delta_i),\]
where $\delta_4=x_1\delta_1+x_2\delta_2+x_3\delta_3,\,x_i \in \mathbb{Q}.$

For the cases where $\delta_4$ does not exist, we stop there. If it exist, then we repeat the same thinking for a new element,
$\delta_5$ and so on.

We will have proved the conjecture to theorem 1.1.2, from \cite{1}, when we find all convex polytopes in the hyperbolic space,
which correspond to the matrices from Table 1 (from \cite{3}).



\begin{thebibliography}{10}

\bibitem{1}
{\sc Viacheslav V. Nikulin}
\newblock {\em Reflection Groups in Hyperbolic Spaces and the Denominator Formula for Lorentzian Kac-Moody Algebras,}
\newblock Izv. Math., 1996, 60 (2), 305-334.

\bibitem{2}
{\sc E.B. Vinberg}
\newblock {\em Hyperbolic Reflection Groups,}
\newblock Russian Math. Surveys 40:1 (1985), 31-75.

\bibitem{3}
{\sc Valeri A. Gritsenko, Viacheslav V. Nikulin}
\newblock {\em Automorphic Forms and Lorentzian Kac-Moody Algebras, Part 1,}
\newblock International Journal of Mathematics, Volume 9, No. 2 (1998).

\bibitem{4}
{\sc James E. Humphreys}
\newblock {\em Introduction to Lie Algebras and Representation Theory,}
\newblock Springer-Verlag (1980).


\bibitem{5}
{\sc Viacheslav V. Nikulin}
\newblock {\em A Lecture on Kac-Moody Lie Algebras of Arithmetic Type ,}
\newblock Preprint alg-geom/9412003.

\bibitem{6}
{\sc Victor G. Kac}
\newblock {\em Infinite Dimensional Lie Algebras,}
\newblock Cambridge University Press, third edition (1990).

\end{thebibliography}
\end{document}